\documentclass[10pt]{amsart}
\usepackage{cite}
\usepackage{amsmath}
\usepackage{amsthm}
\usepackage{amssymb}
\usepackage{amsfonts}
\usepackage{array}
\usepackage{fullpage}
\usepackage{tikz-cd}
\usepackage{enumerate}
\usepackage{multicol}
\usepackage{geometry}
\usetikzlibrary{shapes.geometric, arrows}
\let\emptyset\varnothing

\theoremstyle{plain}
\newtheorem*{theorem*}{Theorem}
\newtheorem*{principle*}{Principle}
\newtheorem{theorem}{Theorem}
\newtheorem{lemma}[theorem]{Lemma}
\newtheorem{proposition}[theorem]{Proposition}
\newtheorem{corollary}[theorem]{Corollary}

\newtheorem{definition}[theorem]{Definition}
\newtheorem*{lemma*}{Lemma}
\newtheorem*{proposition*}{Proposition}
\newtheorem*{corollary*}{Corollary}
\newtheorem*{conjecture*}{Conjecture}
\newtheorem*{definition*}{Definition}
\theoremstyle{remark}
\newtheorem{remark}[theorem]{Remark}
\newtheorem*{ackn}{Acknowledgments}
\theoremstyle{definition}
\newtheorem{example}{Example}[section]

\tikzstyle{startstop} = [rectangle, rounded corners, minimum width=3cm, minimum height=1cm, text centered, text width= 4cm, draw=black]
\tikzstyle{io} = [rectangle, minimum width=3cm, minimum height=1cm, text centered, text width=3cm, draw=black]
\tikzstyle{process} = [rectangle, minimum width=3cm, minimum height=1cm, text centered, text width=4cm, draw=black]
\tikzstyle{decision} = [rectangle, minimum width=3cm, minimum height=1cm, text centered, text width=4cm, draw=black]
\tikzstyle{arrow} = [thick,->,>=stealth]

\def\C{\mathbb{C}}

\def\Q{\mathbb{Q}}

\def\P{\mathbb{P}}

\def\co{\mathcal O}

\def\co{\mathcal{O}}

\def\codim{\operatorname{codim}}
\def\dim{\operatorname{dim}}

\def\X{\bar X}
\def\k{\kappa}

\def\D{{\Delta}}
\def\L{{\Lambda}}

\author{Antonella Grassi and David Wen}
\title{Higher Dimensional Elliptic Fibrations and Zariski Decompositions}
  \address{Antonella Grassi, Department of Mathematics, University of Pennsylvania, Philadelphia, PA 19104 and Dipartimento di Matematica, Universit\`a di Bologna, 40126 Bologna, Italy}
  \email{grassi@math.upenn.edu, antonella.grassi3@unibo.it}
  \address{David Wen, National Center for Theoretical Sciences, No. 1 Sec. 4 Roosevelt Rd., National Taiwan 
  	University, Taipei, 106, Taiwan}
  \email{dwen@ncts.ntu.edu.tw}
 
\begin{document}

\newcommand{\bigslant}[2]{{\raisebox{.2em}{$#1$}\left/\raisebox{-.2em}{$#2$}\right.}}

\thanks{AG gratefully acknowledges the support of a Simons Fellowship. This material is in part based upon work supported by the NSF Grant  DMS-1440140 while the first author was in residence at the Mathematical Sciences Research Institute in Berkeley, California, during the Spring 2019 semester.   AG is a member 
	of GNSAGA of INDAM}

\begin{abstract} 
We study the existence and properties of  birationally equivalent  models for  elliptically fibered  varieties. In particular these have 
either the structure of Mori fiber spaces or, assuming some standard conjectures, minimal models with a Zariski decomposition compatible with the elliptic fibration. 
  We prove relations between the birational invariants of the elliptically fibered variety, the base of the fibration and of its Jacobian.\end{abstract}

\maketitle

\section{Introduction} 



The geometry of elliptic surfaces is well understood by the work of Kodaira.  In particular, when the Kodaira dimension of an elliptic surface is non-negative, the minimal model has a birationally equivalent elliptic fibration. 
Kodaira's canonical bundle formula for relatively minimal elliptic surfaces relates the canonical bundle of the surface to the pullback of the canonical divisor of the base curve and a $\Q$-divisor $\Lambda$ supported on the loci of the image of singular fibers, the discriminant locus  of the fibration. 
 The first author showed   that  the  fibration structure on an elliptic threefold is  compatible with the minimal model algorithm and in addition, that  a  generalization of Kodaira's formula for the canonical divisor holds on the (relative) minimal model \cite{Grassi91,Grassi95}.
 An ingredient  in the proof of  \cite{Grassi91} is to show the existence of  an appropriate combination of the  Zariski Decomposition Theorem for surfaces and a relative version of the minimal model program.    A challenge  in  dimension $4$ (and higher,  which we address here, is the existence of different definition(s) of Zariski decompositions and their relation with  minimal models.
 
  This paper addresses the case of elliptic fibrations of varieties of dimension $\geq 4$.
In the following $\ \Pi: Y \rightarrow T $ is  an elliptic fibration between normal  complex projective varieties where   $\dim Y=n$. Then there exists a birationally equivalent elliptic fibration $\pi: X\to B$,  where $X, \ B$  are smooth and the fibration has nice properties, in particular there exists an effective  $\Q$-divisor $\Lambda$, supported on the discriminant of  the fibration (Theorem \ref{general} and Lemma \ref{HironakaFlat}).



\begin{theorem*}[Proposition \ref{kodEff}, Theorems \ref{MM}, \ref{NewMain}, \ref{KawNakayama}, Corollary \ref{FZD}, \ref{Equidimension}]
	Let $\pi: X \rightarrow B$ be an elliptic fibration between smooth varieties,  $\Lambda $ the discriminant $\Q$-divisor (Lemma  \ref{HironakaFlat}). Then 
	\begin{enumerate}
		\item $\k(X)= \k (B, \Lambda)$.
		\item If $K_B+ \Lambda$ is not pseudo-effective, there exists a birational equivalent fibration $\bar \pi : \bar{X} \to \bar B$, $\bar{X}$ with $\Q$-factorial  terminal singularities, $(\bar B, \bar \Lambda)$ with  klt singularities such that $K_{\bar X}  \equiv \bar{\pi}^*(K_{\bar B} + \bar \Lambda)$. \\
		$X$ is birationally a Mori fiber space.

		\item If $K_B+ \Lambda$ is pseudo-effective, equivalently $\kappa (X) \geq 0$, and klt flips exist and terminate in dimension $n-1$,
there exists a birational equivalent fibration $\bar{X} \to \bar B$, $\bar{X}$ minimal,
			 with  $\Q$-factorial terminal singularities,
			$(\bar B, \bar \Lambda)$ with $\Q$-factorial klt singularities such that $K_{\bar X}  \equiv \bar{\pi}^*(K_{\bar B} + \bar \Lambda)$. 

		\item There is a birationally equivalent fibration  $\bar \pi: \bar{X_1} \to \bar B_1$, with the same properties of $\bar X \to \bar B$ in either (2) or (3) above,  which is 
	equidimensional over an open set $U \subset \bar B$ with $\codim (\bar B \setminus U) \geq 3$.
	
	\end{enumerate}
\end{theorem*}

{To prove part of the theorem above  we show  the  compatibility of a Zariski type decomposition, the Fujita-Zariski decomposition, with the elliptic fibration.  We prove that the compatibility  plays a role in keeping track of the birational modifications of  the steps in the MMP. More specifically, we prove:}

  \begin{theorem*}[Theorem \ref{EllFZD}, \ref{NewFZ}] Let $\pi: X \rightarrow B$ be an elliptic fibration as above  and {$\dim X=n$}. 
	\begin{enumerate}
	\item $K_X$ birationally admits a Fujita-Zariski decomposition if and only if $K_B + \Lambda$ birationally admits a Fujita-Zariski decomposition. 
	\item If $K_B+ \Lambda$ is pseudo-effective,  equivalently if $\kappa (X) \geq 0$,  and klt flips exist and terminate in dimension $n-1$, $K_X$ birationally admits a Fujita-Zariski decomposition compatible with the elliptic fibration structure.
\end{enumerate}
 \end{theorem*}
  
\begin{corollary*}[Corollary \ref{EllFib5}]
Let $\pi: Y \rightarrow T $ be an elliptic fibration, $\kappa(Y)\geq 0$  and $\dim Y \leq 5$. There exists a birational equivalent fibration $\bar{X} \to \bar B$, $\bar{X}$ minimal, 
 $(\bar B, \bar \Lambda)$ with $\Q$-factorial klt singularities such that $K_{\bar X}  \equiv \bar{\pi}^*(K_{\bar B} + \bar \Lambda)$ and 
$K_Y$ birationally admits a Fujita-Zariski decomposition compatible with the elliptic fibration structure.
\end{corollary*}
 \begin{corollary*}[Proposition \ref{kodEff}, Theorem \ref{MM}, Corollary  \ref{ifAbundance}]
Let $\pi: Y \rightarrow T $ be an elliptic fibration
\begin{enumerate}
\item If  $\dim (Y) = 4$ there exists a birationally equivalent fibration $\bar{X} \to \bar B$, $\bar{X}$ with $\Q$-factorial  terminal singularities, $(\bar B, \bar \Lambda)$ with  klt singularities such that $K_{\bar X}  \equiv \bar{\pi}^*(K_{\bar B} + \bar \Lambda)$.
Either 
$Y$ is birationally a Mori fiber space or  $\bar{X}$ is a good minimal model.
\item 
 If $\kappa(Y) = n - 1$, there exists  a birationally equivalent fibration $\bar{\pi}:\bar{X} \rightarrow \bar{B}$ such that  $\bar{X}$ is a good minimal model with $\Q$-factorial  terminal singularities, $K_{\bar{X}} \equiv_{\mathbb{Q}} \bar{\pi}^*(K_{\bar{B}} + \bar{\Lambda})$ and $(\bar{B},\bar{\Lambda})$ has klt singularities.
\end{enumerate}

\end{corollary*}
    In Section \ref{Notations-Results}, we review standard definitions and relevant results about elliptic fibrations, minimal model theory and generalized Zariski decompositions.  We also  highlight the different generalizations of the Zariski Decomposition, their properties and their relationship with minimal model theory and with the structure of elliptic fibrations.
In Section \ref{rmm}, we prove results on the relations between the birational invariants of the elliptically fibered variety, the base of the fibration and of its Jacobian. In particular we have: 
\begin{corollary*}[Corollary \ref{CY}] If  $X$ 
	 is birationally a Calabi-Yau variety, so is $J(X)$.
\end{corollary*}
 Some of these results are applied in the proofs of the theorems in Section \ref{PosKod}. We also prove the relevant parts of the theorems stated above.
In Section \ref{PosKod} we construct a compatible  Zariski decomposition for elliptically fibered varieties of non-negative Kodaira dimension. 
Sections \ref{Existence} and  \ref{Abundance} contain applications, namely we prove existence results and related implications on abundance and other generalized Zariski decompositions for elliptic fibrations. 
Finally, in Section \ref{equid} we prove results  on the dimension of the fibers on special birational models, which we construct. These statements replace the equidimensionality results for dimension $3$.
 In fact, while it is easy to fabricate  examples of minimal elliptic threefolds which are not equidimensional starting from ones which are,  many  smooth Calabi-Yau threefolds have  a natural  elliptic fibration which is not equidimensional. 
Theorem \ref{KawNakayama} and Corollary \ref{equidim3} are stronger than what one could get from a (relative)  log-minimal model run, in at least two aspects. Namely,  the singularities   are terminal, while a minimal model run gives klt singularities, as  in Example \ref{Ex1.1}, but sometimes the model does not have desired properties, as in Example \ref{Ex2}.  In addition, not only are there no exceptional divisors in the fibers outside a codimension $3$ set, but the fibration is actually equidimensional there.\\
 \indent We present applications all throughout the paper.
The techniques of this paper set a foundation for generalization to the case of fibration of Calabi-Yau varieties and as well as log pairs  \cite{DW}. 
Unless otherwise specified, the varieties in this paper are assumed to be complex, projective and normal.

\begin{ackn} DW would like to thank his advisor, Dave Morrison, for helpful discussions and support during his graduate studies, where portions of this work first started. We would also thank R. Svaldi for asking the question we answer in Section \ref{XBJac}. We thank the referee for many helpful
	comments.
\end{ackn}


\section{Notation-Results}\label{Notations-Results}
An elliptic fibration is a morphism, $\pi :X\rightarrow B$ between normal projective varieties, whose general fibers are {smooth} genus one curves with or without a marked point; the complement s{of the image of the smooth fibers} in $B$ is the discriminant of the fibration. If $\pi$ has a section, namely if the general elliptic curve has a marked point, then $X$ is a (smooth) resolution of $W$, the Weierstrass model of the fibration,\cite{Nakayama88}. 

\begin{definition}The elliptic fibrations $\psi:Y \rightarrow T$ and  $\pi$  are  birationally equivalent if there exist birationals maps $f: X \dashrightarrow Y$ and $g: B \dashrightarrow T$ such that the following diagram commutes:
\[
\begin{tikzcd}
\displaystyle
X \arrow[d, "\pi"'] \arrow[r, dashed, "f"]& Y \arrow[d, "\psi"]\\
B \arrow[r, dashed, "g"] & T\\
\end{tikzcd}
\]
\end{definition} 
Building on the work of Kodaira, Kawamata and Morikawi \cite{Kodaira1, Kawamata1,Moriwaki87},  Fujita and Nakayama   proved results in   \cite[Thm 2.14 and Thm 2.15]{Fujita86}  and   in \cite[Thm. 0.2]{Nakayama88}  respectively, which  can be combined to the following:

\begin{theorem}
	\label{general}
Let $\pi: X \to B$ an elliptic fibration  between smooth varieties. 
Then the discriminant locus is a divisor
;  assume that it  has simple normal crossing. In addition:

\begin{enumerate}
\item The $\mathbf{J}$-invariants of the fibers extends to a morphism $ \mathbf{J} : B \rightarrow \P^1$. 
\item $\pi _* (K_ { X/B})  $ is a line bundle. 
\item \label{dis}
$12 \pi _* (K_ { X/B}) =  \co _S \left (\sum 12 a_k D_k \right ) \otimes \mathbf{J} ^* \co _{\mathbb P ^1} (1)$ where  $a_k$ are the rational numbers corresponding to the type
of singularities over the general point of $D_k$, the irreducible components of the ramification locus.   $\mathbf{J}$  has a pole of order $b_k$ along $D_k$,  we write $\mathbf{J} ^* \co _{\mathbb P ^1} (1)=\sum b_ k D_k$. 
\item $K_ X  \equiv _{\mathbb Q} \pi^* \left ( K_B + \pi_* (K_{X /B}) +
\sum \frac {m_i - 1}{ m_i}  Y_i \right ) + E - G$,
 where:
\item The general fiber of $\pi^{-1} (Y_i)$ is a multiple fiber of multiplicity $m_i$
\item  $E|_{ \pi ^{-1}(C)}$ is a union of a finite numbers of proper transforms of exceptional curves, for $C$ a general curve on $B$.
\item G is an effective $\mathbb Q-$divisor and $\codim  \pi (G) \geq 2$,
 \item $\pi_*  \co _X ([mE] )= \co _B  \text{ , } \forall m \in \mathbb{N}$
 \item $\pi ^* \left ( \frac {m_i - 1}{ m_i}  Y_i \right )  + E - G $ is effective.
  \end{enumerate}
\end{theorem}


\begin{definition}\label{def:DeltaLambdasnc}
With the same hypothesis and notation of Theorem \ref{general}, we define the divisors:
 
 $$ \D_{X/B} \stackrel{def}= \sum a_k D_k + \frac{1}{12}J, 
\ \quad \ 
   \L_ { X/B} \stackrel{def}= \D_{X/B} + \sum \frac {m_i - 1}{ m_i}  Y_i 
$$
where  $J$ is an effective divisor whose associated line bundle is equivalent to $\mathbf{J} ^* \co _{\mathbb P ^1} (1)$. When the notation is clear from the context, we will write  $\Delta$ and $\Lambda$. $\D$ and $\L$ are $\Q$-divisor. 
\end{definition}
The pairs $(B, \D)$  and $(B, \L) $ are \emph{klt}, by \cite[Prop. 2.41]{KollarMori}, since $\D$ and $\L$ are simple normal crossing divisors with rational coefficients in $(0,1)$ . More generally, they are log pairs which are of the form $(B,D)$ where $B$ is a normal (projective) variety,  $D_i$ prime divisors, $D = \sum a_i D_i$ and $K_B+ D$ $\Q$-Cartier. 

\begin{definition} \label{plt}
 \begin{enumerate}
\item [\emph{(i):}] A \underline{\emph{log resolution}} of $(B, D)$ is  a resolution  $f: {\widetilde B }\to B$,  such that the union
of  $\sum a_i f^{-1}_*(D_i)$, the strict transform of $D$, and
 the exceptional locus of $f$ are supported on divisors with simple normal
 crossings.
\item [\emph{(ii)}] We then  write
 $K_{\widetilde B} + \sum a_i f^{-1}_*(D_i)= f^*( K_B + D)+ \sum a(E_j,X, D)E_j,$ where $ a(E_j,X, D)$ are the discrepancies.
\item [\emph{(iii)}] The pair $(B, D)$ is
\begin{enumerate}
 \item[\emph{terminal:}]  
 if for any (equivalently for every) log resolution $f$, $a(E_j,X, D) > 0, \forall j$ .
\item[\emph{klt:}]
 if for any (equivalently for every) log resolution $f$, $a(E_j,X, D) >-1, \forall j$ .
\item[\emph{lc:}] ({log canonical}) if for any (equivalently for every) log   resolution $f$,  $a(E_j,X, D)\geq -1, \forall j$.
    \item[\emph{dlt:}]  (divisorially log terminal) if there is
     a log resolution $f$
such that $a(E_j,X, D) >-1 $ for every exceptional divisors $E_j$.
\item[\emph{plt:}] ({purely log terminal}) if for any log resolution $f$,  $a(E_j,X, D) >-1$, for every coefficient of an exceptional divisor $E_j$.
\end{enumerate}
 \end{enumerate}
 \end{definition}
 In the following $(X,D)$ is always a lc pair.

\begin{definition}\label{lmm} Minimal models, log minimal models etc.
\begin{enumerate}
\item[\emph{MM:}]  $\X$ is a minimal model if $(\X, 0)$ has  terminal singularities, $K_{\X}$ is nef and $\X$ is $\Q$-factorial.
\item [\emph{Neg. Contr.:}]  $\psi: B \dasharrow \bar B$ is a $(K_B + D)$-negative contraction if\\
  $\psi^{-1}$ does not contract any divisor and there exists a resolution $\tilde B$
\[
\begin{tikzcd}
&\tilde{B} \arrow[ld, "g"'] \arrow[rd, "h"] &\\
(B, D) \arrow[rr, dashed, "\psi"]&& (\bar{B}, \psi_*(D))\\
\end{tikzcd}
\]
such that $g^*(K_B+ D)- h^*(K_{\bar B}+ \bar {D})= \sum a_jE_j, \ \ a_j >0$ and $E_j$ exceptional for $h$.
\item [\emph{LMM-A:}]  $(\bar B, \psi_*(D))$ is a  log minimal model for $(B, D)$ if\\
 $\psi$ is a $(K_B + D)$-negative contraction and $(K_{\bar B} + \psi_*(D))$ is nef.
\item [\emph{LBM:}] $(\bar B, \bar D))$ is a  log birational model of $(B, D)$ if  \\
 $\psi: B \dasharrow \bar B$ is birational and $\bar D \stackrel{def}= \psi_*(D) +E$, where $E$ is the reduced exceptional divisor of $\psi^{-1}$.
\item [\emph{LMM-B:}]   (\cite{Birkar1} ) A log birational model  $(\bar B, \bar D )$ is a  log minimal model for $(B, D)$ if\\
 $(\bar B, \bar D)$ is $\Q$-factorial dlt, $(K_{\bar B} + \bar D))$ is nef and $a(E_j,B, D) < a(E_j,\bar B , \bar D) $, for $E_j$ divisor in $B$, exceptional for $\psi$.
\item [\emph{GOOD}]  A log minimal model $(\bar B, \psi_*(D))$ is good if $K_{\bar B} + \psi_*(D)$ is semi-ample.
\end{enumerate}
 \end{definition}

\begin{remark} The definition LMM-B  allows for $\psi^{-1}$ exceptional divisors. Furthermore, the Negativity Lemma implies that a log minimal model according to A is a log minimal  model in the sense of B; the two definitions are equivalent for plt pairs \cite{Birkar1}.
\end{remark}

\begin{definition}[Zariski Decompositions \cite{Birkar1} - \cite{Fujita86} - \cite{CutKM, CKawM, CKMor} - \cite{Nakayama2}]\label{ZD}
Let $X$ be a normal, projective variety with a proper map $\pi:X \rightarrow Z$ and $D$ a $\mathbb{R}$-divisor on $X$. We have that $D = P + N$ is called:
\begin{enumerate}
\item[\emph{W:}] A Weak Zariski decomposition over $Z$, if $P$ is $\pi$-nef and $N$ is effective.
\item[\emph{FZ-A:}] A Fujita-Zariski decomposition over $Z$, if it is a Weak Zariski decomposition and we have that for every projective birational morphism $f:W \rightarrow X$, where $W$ is normal, and $f^*D = P' + N'$ with $P'$ nef over $Z$, then we have $P' \leq f^*P$.
\item[\emph{CKM:}] A CKM-Zariski decomposition over $Z$, if it is a Weak Zariski decomposition and we have that $\pi_*\mathcal{O}_X(mP) \rightarrow \pi_*\mathcal{O}_X(mD)$ is an isomorphism for all $m \in \mathbb{N}$.
\end{enumerate}
If we have that $Z = Spec(\mathbb{C})$ then we will refer to $D = P + N$ as simply the (Weak, Fujita, CKM) Zariski decomposition. Additionally, for the case where $Z = Spec(\mathbb{C})$ and $X$ smooth, we have the following original definition of the Fujita-Zariski decomposition.
\begin{enumerate}
\item[\emph{Num. Fixed:}] Let $E$ be an effective $\mathbb{Q}$-divisor and $L$ be a $\mathbb{Q}$-divisor on $X$. We say $E$ clutches $L$ if, for any effective $\mathbb{Q}$-divisor $F$ where $L - F$ is nef, we have that $E - F$ is effective. We say $E$ is numerically fixed by $L$ if for any birational morphism $\pi: W \rightarrow X$ we have that $\pi^*E$ clutches $\pi^*L$.
\item[\emph{FZ-B:}]  $D = P + N$ is a Fujita-Zariski decomposition if $N$ is numerically fixed by $D$.
\end{enumerate}
Additionally, if we also assume that $D$ is pseudo-effective we have the sectional decomposition (sometimes called the Nakayama-Zariski decomposition).
\begin{enumerate}
\item[\emph{NZ:}] Let $A$ be a fixed ample divisor on $X$. Given a prime divisor $\Gamma$ on $X$, define
\[
\sigma_\Gamma(D) = \min \{ \text{mult}_\Gamma(D') | D' \geq 0, D' \sim_{\mathbb{Q}} D + \epsilon A \text{ for some } \epsilon > 0 \}
\]
This definition is independent of the choice of $A$. Furthermore, it was also shown in \cite{Nakayama2} that for only finitely many $\Gamma$ that $\sigma_\Gamma(D) > 0$. This allows us to define the following decomposition.\\
\item[] Let $N_\sigma(D) = \sum_\Gamma \sigma_\Gamma(D) \Gamma$ and $P_\sigma(D) = D - N_\sigma(D)$, then we call $D = P_\sigma(D) + N_\sigma(D)$ the sectional decomposition. If we have also that $P_\sigma(D)$ is nef then we refer to this as the Nakayama-Zariski decomposition of $D$.
\end{enumerate}
We say $D$ birationally admits a (Weak, Fujita, CKM, Nakayama) Zariski decomposition over $Z$ if there exists some resolution $f: Y \rightarrow X$ such that $f^*(D)$ has a (Weak, Fujita, CKM, Nakayama) Zariski decomposition over $Z$.
\end{definition}

\begin{remark}
\label{ZariskiRelation}
There is a nesting of the above generalized Zariski decompositions as listed:
\begin{enumerate} 
	\item A Nakayama-Zariski decomposition ( a sectional decomposition with nef positive part) is a Fujita-Zariski decomposition.
	\item A Fujita-Zariski decomposition is a CKM-Zariski decomposition.
	\item These are all Weak Zariski decompositions.
	\item There are CKM-Zariski decompositions that are not Fujita-Zariski decompositions.
	\item It is not known if there are Fujita-Zariski decompositions that are not Nakayama-Zariski decompositions.
\end{enumerate}
Below we list some technical properties, relations and similarities of the different versions of the generalized Zariski decompositions.
\end{remark}

\begin{proposition}[{\cite[Cor. 1.9; Lemma 1.22]{Fujita86}} ]
\label{NumFix}
Let $X$ be a smooth projective variety with $E$ an effective $\mathbb{Q}$-divisor that is numerically fixed by a Cartier divisor $L$.
\begin{enumerate}
	\item Let $F$ be the smallest Cartier divisor such that $F - E$ is effective, then we have the following isomorphism of graded rings:
	\[
	\bigoplus_{t \geq 0} H^0(X, tL) \equiv \bigoplus_{t \geq 0} H^0(X, tL - tF)
	\]
	\item  $L - E$ admits a Fujita-Zariski decomposition if and only if $L$ admits a Fujita-Zariski decomposition. Additionally, the nef parts of the decompositions are the same.
\end{enumerate}
\end{proposition}

\begin{proposition}[{\cite[Prop. 1.10]{Fujita86} - \cite[Lemma 2.16]{GongyoLehmann}}]
Let $f:M \rightarrow S$ be a surjective morphism of manifolds with connected fibers. Let $X$ be a divisor on $M$ such that $\dim f(X) < \dim S$. Suppose that for every irreducible component $Z$ of $f(X)$ with $\dim Z = \dim S - 1$, there is a prime divisor $D$ on $M$ such that $f(D) = Z$ and $D \not\subset Supp(X)$. 
\begin{enumerate}
	\item $X$ is numerically fixed by $X + f^*L$ for any $\mathbb{Q}$-Cartier divisor on $S$.
	\item For any pseudoeffective $\mathbb{R}$-divisor $L$ on $S$,  $D \leq N_\sigma(f^*L + D)$ and $P_\sigma(f^*L + D) = P_\sigma(f^*L)$.
\end{enumerate}
\end{proposition}

\begin{proposition}[{\cite[Prop. 1.24]{Fujita86}}]
\label{ExToZariski}
Let $:M \rightarrow S$ be a surjective morphism of manifolds with $L$, a $\mathbb{Q}$-Cartier divisor, on $S$ and $R$ an effective $\mathbb{Q}$-divisor on $M$ such that $\dim f(R) \leq \dim S - 2$. Then  $f^*L + R$ birationally admits a Fujita-Zariski decomposition if and only if $L$ birationally admits a Fujita-Zariski decomposition.
\end{proposition}

\begin{proposition}[{\cite[Prop. V.1.14]{Nakayama2}}]
Let $D$ be a pseudo effective $\mathbb{R}$-divisor, then
\begin{enumerate}
	\item $N_\sigma(D) = 0$ if and only if $D$ is movable.
	\item If $D - E$ is movable for an effective divisor $E$, then  $N_\sigma(D) \leq E$.
\end{enumerate}
\end{proposition}

\begin{proposition}
When $Z = Spec(\mathbb{C})$ and $X$ is  smooth, the two definitions of the Fujita-Zariski decomposition are equivalent.
\end{proposition}
\begin{proof} 
Let $D = P + N$ be a Fujita-Zariski decomposition in the sense of \emph{FZ-B}. We will show that this implies the properties of \emph{FZ-A}. Let $f:X' \rightarrow X$ be a birational morphism with $f^*(D) = P' + N'$ where $P'$ is nef and $N'$ is an effective $\mathbb{Q}$-Cartier divisor. We have that $N$ is numerically fixed by $D$ and so $f^*(N)$ clutches $f^*(D)$. As $f^*(D) - N' = P'$ is nef, we have that $N' - f^*(N)$ is effective. But we know that $N' = f^*(D) - P'$ and $f^*(N) = f^*(D) - f^*(P)$. So by replacing and simplifying we have that $f^*(P) - P'$ is effective.

Let $D = P + N$ be a Fujita-Zariski decomposition in the sense of \emph{FZ-A} and we will show that $N$ is numerically fixed by $D$, so given a birational morphism $f:X' \rightarrow X$ we will show that $f^*(N)$ clutches $f^*(D)$. Thus given an effective $\mathbb{Q}$-divisor $N'$ such that $P' := f^*(D) - N'$ is nef, we want to show that $N' - f^*(N)$ is effective. We can assume that $X'$ is normal, otherwise we can resolve singularities to get $\pi:W \rightarrow X'$, where we have that showing $\pi^*(N' - f^*(N))$ is effective is sufficient to show that $N' - f^*(N)$ is effective on $X'$. Thus without loss of generalities, we can assume $X'$ is normal. Since $P + N$ is a Fujita-Zariski decomposition in the sense of \emph{FZ - A}, we have that $f^*(P) - P'$ is effective. Replacing with $f^*(P) = f^*(D) - f^*(N)$ and $P' = f^*(D) - N'$, we get that $N' - f^*(N)$ is effective. This completes the argument that shows the two definitions are equivalent.
\end{proof}


\begin{remark}
\label{FZAB}
\begin{enumerate}
	\item The original definition of the Fuijita-Zariski decomposition in \cite{Fujita86} is equivalent to our definition of a divisor birationally admitting a Fujita-Zariski decomposition.
	\item If $K_B+ \D $ has a log minimal model, then it birationally has a Fujita (also CKM and Weak) Zariski decomposition. This is shown explicitly as part of the argument of \cite[Thm 1.5]{Birkar1}.
	\item A Fujita-Zariski decomposition and a Nakayama-Zariski decomposition of a divisor is unique. A CKM-Zariski and Weak Zariski decomposition of a divisor need not be unique.
	\item Each of the above generalized Zariski decomposition for the canonical divisor has a different role in birational geometry and their relations to minimal models. The Nakayama-Zariski decomposition is more attune to work with abundance and good minimal models as seen in \cite{GongyoLehmann}. The Fujita-Zariski decomposition aligns with minimal models as seen below, and the CKM-Zariski decomposition corresponds to working with the canonical ring and, as a result, the canonical model. 
	\item  Recent work in \cite{HaconMoraga, HanLi} and \cite{Birkar1}  show that there is a strong correlation between the existence of log-minimal models and Zariski decompositions.
\end{enumerate}
\end{remark}

\begin{theorem}[\cite{Shoku92,Kollar92,Shoku03} -\cite{Kawamata3kltTerm,Kollar92} - \cite{BCHM} - \cite{Birkar2}]
\label{kltmin}
We have the following results in the theory of minimal models.
\begin{enumerate}
	\item Flips for klt pairs exists in all dimension.
	\item Any sequence of klt flips terminate in dimension $3$.
	\item A klt pair in dimension up to $4$ either admits a minimal model or is birational to a Mori Fiber space.
	\item A klt pair, $(X,\Delta)$, such that $\dim(X) = 5$ and $\kappa (X,\Delta) \geq 0$ admits a minimal model.
	\item The abundance conjecture holds for klt pairs of dimension $ \leq 3$. Thus klt pairs of dimension up to $3$ admit a good minimal model or are birational to a Mori Fiber space.
	\item General type klt pairs admit a good minimal model.
\end{enumerate}
\end{theorem}


\medskip

\section{$X \& B$, Relative minimal models, The canonical bundle formula, Jacobians}\label{rmm}

\medskip

\subsection{${\bf{X \& B}}$ and the discriminant} 



We recall the following application of Hironaka's flattening theorem:
\begin{lemma}[\cite{Hironaka75}]\label{HironakaFlat} Let $\Pi: Y \rightarrow T$ be an elliptic fibration between normal varieties. Then there exist  birational equivalent fibrations 
 \[
\begin{tikzcd}
 Y\arrow[d, " \Pi"]  &  X_0\arrow[l]\arrow[d, "\pi_0"]  & X_1\arrow[l]\arrow[d, "\pi_1"] & X \arrow[l]\arrow[d, "\pi"]\\
T   & B_0 \arrow[l]& B_1    \arrow[l]& \arrow[l]   \arrow[bend left=26,swap]{ll}{\psi_0}   \arrow[bend left=35,swap]{lll}{\psi} B  \\
\end{tikzcd}
\]
where $X_0$, $ X$,  $B_0$, $B$ and $B_1$ are smooth, $\pi_1$ is flat and $\pi: X\to B$  satisfy the hypothesis of Theorem  \ref{general}. 
\end{lemma}

Without loss of generality we can also assume $B_1=B$ in the above Lemma.

\begin{definition}\label{def:DeltaLambda}With the notation of  Lemma \ref{HironakaFlat} we set:\\
		$ \Lambda_{B_0} \stackrel{def}= {\psi_0}_* (\Lambda_{X/B})$  and  $ \Lambda_T \stackrel{def}= \psi_* (\Lambda_{X/B})$, where $\Lambda_{X/B}$ is as  in Definition \ref{def:DeltaLambdasnc}.
\end{definition}

In particular the type of the singular fiber over a general point in the codimension $1$ locus of the support of $\Lambda_r$ is determined by the coefficient of its irreducibile component  in $\Lambda_T$. 

\begin{proposition}\label{kodEff} Let $X_0 \to  B_0$ be an elliptic  fibration between smooth varieties and
	$X$ and $B$ be  as in  Lemma \ref{HironakaFlat}. 
	Then $\k(X) = \k(K_{B}+\Lambda) $.
	
\end{proposition}
\begin{proof} The statement holds for $\dim X_0=2$. When $\dim X_0 =3$, from the proof of \cite[Theorem 3.2]{Fujita86} we can deduce  that 
	$\kappa(X)=\kappa(X,K_X+G)=\kappa (X, \pi^*(K_B + \Lambda)+ E)= \kappa (X, \pi^*(K_B+ \Lambda ))=\kappa (B,  K_B+ \Lambda)$, as in \cite[Proposition 1.3]{Grassi91}.
	In particular $h^0(mK_X)=h^0(m(K_B+ \Lambda))$, for all $m > > 0$. The arguments about the pluricanonical rings in the proof of \cite[Theorem 3.2]{Fujita86}  and  \cite[Proposition 1.3]{Grassi91} are independent of dimension and they can be extended to $\dim X_0 \geq 4$.

\end{proof}

We also have: 
\begin{proposition}\label{Prop2.2} 
	Let $\pi :X \rightarrow B $  be an elliptic fibration between manifolds, with $\dim X= n$ and let $\Delta$ the component of the discriminant divisor associated to the invertible sheaf $\pi _* (K_ { X/B})  $.
	
	The following exact sequences and isomorphisms hold:
	$$0\rightarrow H^\ell (B, \co _B ) \rightarrow H^\ell (X,  \co _X ) \rightarrow H^ {\ell -1} (B, -\Delta) \rightarrow 0 , \  \ \ 1 < \ell < n-1 $$
	$$  H^0 (X,  \co _X )   \simeq  H^0 (B,  \co_B),  \ \ \ H^0 (X,  K _X )   \simeq  H^0 (B,  K _B + \Delta ) .$$
\end{proposition}
\begin{proof}  
Recall that  $R ^1 \pi _* ( \co  _X) \simeq  \pi _* ( K _{X/S} ) \simeq \co_B (\Delta) $ is the effective divisor of the discriminant corresponding to the non-multiple fibers (Thereom \ref{general}). The proof is in \cite[Prop.2.2]{Grassi91}. 
\end{proof}

\medskip

\subsection{Relative minimal models}


\begin{theorem}\label{RelMM} Let $\pi : X \to B$ be a an elliptic fibration between  $\Q$-factorial varieties. Assume that $\codim Sing(X) \geq 3 $
and  that 
$K_X \equiv_{\Q} \pi^*(L )+F$, where $F$ is  a $\Q$-effective $\Q$-divisor such that no irreducible component  $F_j$ of $F$ is $F_j= \pi^*(\Gamma_j)$, for some $\Gamma_j$.
 Then
\begin{enumerate}
\item[\emph{(i)}]  $K_X$ is not $\pi$- nef.
\item[\emph{(ii)}]  If in addition $X$  has  terminal singularities there is a  relative good minimal model  for $X$ over $B$, that is, there exists a birational equivalent elliptic fibration ${\pi_r}: {X_r} \to B$, such that  ${X_r}$ has $\Q$-factorial terminal singularities,
$K_{{X_r}}$ is ${\pi_r}$-nef and $\pi_r$-semiample. In addition
$K_{{X_r}} \equiv_{\mathbb{Q}} {\pi_r}^*(L)$.
\item[\emph{(iii)}] $X_r$ be as in (ii). There exists a $\Q$-divisor $\Lambda_r$ such that $L =K_B +  {\Lambda_r}$, ${\Lambda_r}$ defined as in Definition \ref{def:DeltaLambda} and $(B, {\Lambda_r})$ has klt singularities.
\end{enumerate}
 \end{theorem}

  
\begin{proof} Note that the locus of terminal singularities has codimension $\geq 3$.
\begin{enumerate} 
\item[{(i)}] Let $F_j$ be an irreducible component of $F$. If $\codim \pi (F_j) \geq 2$, then there is an effective curve $\gamma$ such that $ 
F_j \cdot \gamma<0$. 
If $\codim \pi (F_j) =1$, let $C$ be a general curve in $B$, $p$ a point in $C$ and in  the support of $\pi (F_j)$.  Consider the 
elliptic surface  $S\stackrel{def}= \pi ^*{C}, \ S \to C$
 and conclude that there is an effective (exceptional) curve $\gamma$ in the fiber over $p$ such that $K_X \cdot \gamma= F_j \cdot \gamma <0$.
\item[{(ii)}] 
Since the $\pi$-relative log canonical ring  is finitely generated (\cite{BCHM}, \cite[Theorem 6.6]{Kawamata2009}),  the hypothesis of Theorem 2.12 in \cite{HaconXu} (which generalizes \cite{Lai2011}) are satisfied. Then there exists 
a  relative good minimal model  for $X$ over $B$, that is, there exists a birational equivalent elliptic fibration ${\pi_r}: {X_r} \to B$, such that  ${X_r}$ 
$K_{{X_r}}$ is ${\pi_r}$-nef and $\pi_r$-semiample.  In particular the proof of  \cite{HaconXu}[Theorem 2.12]   shows that if $X$ has $\Q$-factorial terminal singularities, so does $X_r$.


 We now assume without loss of generality that  there is a birational morphism $\mu: X  \to {X_r}$, with $K_X \equiv \mu ^*(K_{X_r})+ \sum a_i E_i$, $E_i$ $\mu$-exceptional and $a_i >0$.
 We then have 
  $\mu^*(K_{{X_r}} - {\pi_r} ^*(L))= F-\sum a_i E_i$. 
 If $F$ is not $\mu$-exceptional, we can take  $\gamma$ as in part (i) and conclude by contradiction, since\\
  $0 \leq \mu^*(K_{{X_r}} - {\pi_r} ^*(K_B + L))\cdot  \gamma = (F-\sum a_i E_i) \cdot \gamma <0$.
\item[{(iii)}]  It follows from Lemma \ref{HironakaFlat} and \cite[Theorem 0.4]{Nakayama88}.
\end{enumerate}
\end{proof}

\begin{proposition}\label{eff} Let $\pi_i :  X_i  \to B \ i=1,2 $ be birationally equivalent elliptic fibrations, $X_i$ with  terminal singularities,  $(B, \L)$ with klt singularities and
$$K_{X_i} \equiv \pi_i ^*(K_{B} + \Lambda)+ F_i , \  F_i \ \Q\text{-Cartier } \Q \text{-divisor}.$$ 
Then and $F_1$ is  $\Q$-effective if and only if  $F_2$ is.
\end{proposition}
\begin{proof}   We take a common resolution of $X_1$ and $X_2$ and conclude by the  same argument as in the proof of   \cite[Lemma 1.5]{Grassi95}.
\end{proof}
\begin{proposition}\label{Cor1.4} Let $\pi: X \to B $ be  an  elliptic fibrations, $X$ with  terminal singularities,   $(B, \L)$ with klt singularities  and
$K_{X} \equiv \pi ^*(K_{B} + \Lambda) + F$, for some $\Q$-Cartier $\Q$-divisor $F$.\\
Then  $F$ is  $\Q$-effective if and only if  there exists a birational equivalent elliptic fibration $\pi_r: {X_r} \to B$, $X_r$ with terminal singularities, such that 
$K_{{X_r}} \equiv_{\mathbb{Q}}{\pi_r}^*(K_{B} + \Lambda)$.
\end{proposition}
\begin{proof} The same argument as in the proof of   \cite[Corollary 1.4]{Grassi95} applies.
\end{proof}
Note that part (iii) of 
Theorem \ref{RelMM} assure that if $F$ is $\Q$-effective,   $(B, \L)$ is klt. 
In particular no birationally equivalent elliptic fibration $\pi_r: {X_r} \to B$, $X_r$ with terminal singularities, such that 
$K_{{X_r}} \equiv_{\mathbb{Q}}{\pi_r}^*(K_{B} + \Lambda)$  can exist   in the following example:

\begin{example}[Terminal versus klt] \label{Ex1.1} Example 1.1 of \cite{Grassi95} provides  an elliptic fibration $\pi: X \to \P^2$, $X$ smooth, such that $K_X \equiv \pi ^*(K_B+ \Lambda )-\frac{1}{3} D$, where   $D$ is   an effective divisor. The discriminant of the elliptic fibrations consists of $2$ different lines corresponding to multiple fibers of type $_3I_0$.  By Proposition \ref{eff}  there is no birationally equivalent elliptic fibration
	$\phi: \bar X \to \P^2$, $\bar X$ with \textit{terminal singularities}, such that $K_{\bar X} = \phi ^*(L)$, for some $L$. 
However, i this particular example $K_X + \alpha D$ is not relatively nef over $\P^2$ if  $\frac{1}{3}<\alpha <1 $ and we  can explicitly apply the  relative log minimal model 
Theorems of 
\cite{HaconXu} and \cite{Lai2011} 
for the  klt pair $(X, \alpha D)$  to  obtain  a  birational  equivalent elliptic fibration $ \bar \pi : X_r  \to \P^2$ with $X_r \equiv {\bar \pi}^*(K_{B} + \Lambda)$.
$X_r$ has \textit{klt singularities}.
\end{example}

We stress that  it is not guaranteed that running the relative minimal model program on the klt pair $(X, \alpha D)$ would contract $D$, as in the example below, and give a birationally equivalent model  $\bar \phi: X_r \to B$, $X_r$ with klt singularities such that $K_{X_r}={\bar \phi} ^*(L)$,
contrary to  the claim in  \cite[Section 8]{KollarEllipticSurvey}.

\begin{example}[MM run to klt singularities does not guarantee  a  pullback formula for the canonical divisor] \label{Ex2}
Consider the 
 elliptic threefold, $f: X \rightarrow \mathbb{C}^2$,
 $y^2 = x^3 + (s^2-t^2)^4x + s^6t^6$ over $(s,t) \in \mathbb{C}^2$.
$X$ is a local 
Weierstrass model over a smooth surface.
$X$ has a log canonical singularity; in fact, 
 there is a resolution $g:Y \rightarrow X$ such that $K_Y = g^*(K_X) - D = g^*f^*(L) - D$, for some divisor $L$ on $\mathbb{C}^2$ with $D \subset Y$ being a surface over the origin in $\mathbb{C}^2$. As $D$ maps to a point on $\mathbb{C}^2$ we have that $-D$ is nef over $\mathbb{C}^2$. Thus for $(Y, \alpha D)$, there is no value of $\alpha \in [0,1]$ which when running the minimal model program over $\mathbb{C}^2$ would contract $D$, and obtain a birationally equivalent model  $\bar \phi: X_r \to \C^2$,  $X_r$ with klt singularities such that $K_{X_r}={\bar \phi} ^*(L)$.
\end{example}

In  Theorem  \ref{KawNakayama} we prove a more general and precise statement for the context of the above Examples \ref{Ex1.1} and \ref{Ex2}.

\begin{corollary} \label{relNoMf}
Let $\pi: X_0 \rightarrow B_0 $ be an elliptic fibrations between manifolds such that the ramification locus has simple normal crossing as in Theorem \ref{general}.  Assume that either $\pi$ is equidimensional or there are no multiple fibers. 
Then there exists a good minimal model ${X_r}$ of $X$ over $B$, that is a birational map $\mu: X \dashrightarrow {X_r}$ and a morphism ${\pi_r}: {X_r} \rightarrow B$ such that the diagram commutes, $K_{{X_r}}$ is ${\pi_r}$-nef, $K_X \equiv_{\mathbb{Q}}{\pi_r}^*(K_B + \Lambda)$ and ${X_r}$ has terminal singularities.
\end{corollary}
\begin{proof} In fact  if there are no multiple fiber  $F= E-G$ is effective; if $\pi$ is equidimensional then  $F=E$ is  effective and we conclude by Proposition \ref{Cor1.4}
\end{proof}

\medskip

\subsection{Birational equivalent elliptic fibrations, minimality, Mori fiber spaces and the canonical bundle formula}

\begin{proposition}\label{canF} Let $\pi: X \to B$ an elliptic fibration  between manifolds. Assume that the ramification divisor of the fibration has simple normal crossing as in Theorem \ref{general}. Then,  there is a birationally equivalent elliptic fibration $\pi_r: X_r \to B_r$
such that $X_r$ has terminal singularities,  $(B_r, \Lambda_r)$ has klt singularities and  $K_{X_r} \equiv_{\Q} \pi_r^*(K_{B_r}+ \Lambda_r)$. \end{proposition}

\begin{proof} There is a relative good minimal model $\pi_{r'}: X_r \to B$, by  the proof of Theorem 2.12 in \cite{HaconXu} and \cite{BCHM}. Note that  \cite{HaconXu} generalizes \cite{Lai2011}).
 In particular there exist a birational  morphism  $\phi: B_r \to B$, a birationally equivalent fibration $\pi_r:  X_r \to B_r$  and  $L_r$  $\phi$-semiample such that $K_{X_r} \equiv  \pi_r ^*(L)$ and $\pi_{r'}= \phi \cdot \pi_r $.
  Note that $\dim B_r \geq \dim B$, that   the morphism $\phi$ has to be birational; $\pi_r$ is a birationallly equivalent elliptic fibration. $X_r$  has terminal singularities and then 
$K_{X_r}=\pi_r^*(K_{B_r}+\Lambda_r)$ and$ (B_r,\Lambda_r)$ is klt \cite[Corollary 0.4]{Nakayama88}.

\end{proof}
Examples \ref{Ex1.1} and \ref{Ex2} show that it can be necessary to birationally modify the base: $B_r \to B$.
\begin{corollary} \label{MMb}
Let $\pi: X \rightarrow B $ be an elliptic fibration between manifolds such that the ramification locus has simple normal crossing as in Theorem \ref{general}.  
\begin{enumerate}
\item If $K_{B} + \Lambda$ is not  pseudo effective, there exists a birational equivalent fibration $\bar{X} \to \bar B$, $\bar{X}$ with terminal singularities, $(\bar B, \bar \Lambda)$ with klt singularities such that $K_{\bar X}  \equiv \bar{\pi}^*(K_{\bar B} + \bar \Lambda)$. In  addition $X$ is birationally a Mori fiber space.
\item If $K_{B} + \Lambda$ is pseudo effective  and klt flips exist and terminate in dimension $n-1$, then there exists a birational equivalent fibration $\bar{X} \to \bar B$, $\bar{X}$ minimal, with terminal singularities, $(\bar B, \bar \Lambda)$ with klt singularities such that $K_{\bar X}  \equiv \bar{\pi}^*(K_{\bar B} + \bar \Lambda)$.
\end{enumerate}
\end{corollary}
\begin{proof}  
Let $\pi_r: {X_r} \to B_r$ a birationally equivalent elliptic fibration as in Proposition \ref{canF}.
If $K_{B} + \Lambda$ is not  pseudo effective, then by  \cite[Corollary 1.3.2]{BCHM} $(B, \Lambda)$ is birationally a Mori fiber space, that is there exist a is a $(K_B+ \Lambda))$-negative birational contraction $ \psi: B  \dasharrow  \bar B$ and a morphism $f : \bar B \to Z $ with connected fibers such that $dim(Z) < dim(B)$ and ${(K_{\bar B} + \psi_*\Lambda)}_{|F}$ is anti-ample for a general fiber F of $f$.   We then conclude by applying  Corollary 2.13 in \cite{HaconXu} to every birational contraction and flip in $\psi$. This shows (1). Part (2) follows similarly. \end{proof}

\begin{theorem} \label{MM}
Let $\pi: Y \rightarrow T$ be an elliptic fibration 
\begin{enumerate}
\item If $\kappa (Y) \geq 0$ and klt flips exist and terminate in dimension $n-1$, then there exists a birational equivalent fibration $\bar{X} \to \bar B$, $\bar{X}$ minimal, with terminal singularities, $(\bar B, \bar \Lambda)$ with klt singularities such that $K_{\bar X}  \equiv \bar{\pi}^*(K_{\bar B} + \bar \Lambda)$.
\item Let be $\Lambda_T$ the $\Q$-divisor defined in Definition \ref{def:DeltaLambda}.
 If $K_T+ \Lambda_T$ is not pseudo-effective 
 there exists a birational equivalent fibration $\bar{X} \to \bar B$, $\bar{X}$ with terminal singularities, $(\bar B, \bar \Lambda)$ with klt singularities such that $K_{\bar X}  \equiv \bar{\pi}^*(K_{\bar B} + \bar \Lambda)$. In  addition $X$ is birationally a Mori fiber space.
\end{enumerate}
\end{theorem}
\begin{proof}
It follows from Corollary \ref{MMb} and Proposition \ref{kodEff}.
\end{proof}

\medskip

\subsection{ Jacobians}\label{XBJac}~

\smallskip

Let $\pi :X \rightarrow B $  be an elliptic fibration between manifolds, with $\dim X= n$ and   the ramification divisor  a divisor with simple normal crossings . 
The corresponding Jacobian elliptic fibrations $\pi_J: J(X) \to B$   is defined birationally \cite{DolgachevGross} from the relative minimal model of  $X \to B$, which exists by Theorem \ref{RelMM}. 

\begin{proposition}\label{XJX} Let $\pi :X \rightarrow B $  be an elliptic fibration between manifolds, with $\dim X= n$ and 
	the ramification divisor a divisor with simple normal crossings . Let $\pi_J: J(X) \to B$ be the Jacobian fibration as above. Then:
	$$h^i(X, \mathcal{O}_X)=h^i(J(X), \mathcal{O}_{J(X)}),  \ 1< j < n-1.$$
\end{proposition}

\begin{proof}  In fact  it follows from \cite[Proposition 2.17]{DolgachevGross} that the ramification divisor  of  the Jacobian fibration  is  a simple normal crossing divisor. In addition $\Delta=\Delta_J$.
	Proposition \ref{Prop2.2} implies the statement.
\end{proof}

\begin{proposition}\label{ko} With the same hypothesis of Proposition \ref{XJX},
	assume also that $h^0(X,K_X)=1$ and $\kappa(X)=0$, then  
	$h^0(J(X), K_{J(X)})=1$ 
	and $\kappa(J(X))=0$,
\end{proposition}

\begin{proof}  Since $\Delta=\Delta_J$,  Proposition \ref{Prop2.2}  implies that $h^0(J(X), K_{J(X)})=h^0(X, K_X)=1$. Furthermore,   as in the proof of Proposition \ref{kodEff},  for $\ \ m >>0$,
	$$h^0(X, mK_X)=h^0(B, m(K_B+ \Lambda)) \geq h^0(B, m(K_B+ \Delta))= h^0(B, m(K_B+ \Delta_J))=h^0(J(X), mK_{J(X)}). $$
	
\end{proof}

\begin{corollary}\label{CY} If  $X$ has  birationally a trivial canonical divisor and $h^i(X, \co_X)=0, \ 0 <i <n$, that is if $X$ is birationally a Calabi-Yau variety, so is $J(X)$.
\end{corollary}

\begin{proof} The statement follows from Proposition \ref{Prop2.2} and \ref{ko}.
\end{proof}

\section{Non negative Kodaira dimension, minimal models, Zariski decomposition and the canonical bundle formula}\label{PosKod}

We prove,  assuming  standard conjectures in the theory of minimal models, 
a birational Fujita-Zariski decomposition for the canonical divisor for elliptic fibrations with non-negative Kodaira dimension. 
 We use  properties of the two definitions of the Fujita-Zariski decomposition.  From
\cite{Birkar1}, we use the relationship between Fujita-Zariski decomposition and minimal model theory; from \cite{Fujita86}, we use the relationship between Fujita-Zariski decomposition and the properties of numerically fixed divisors. 
This, in conjunction with the canonical bundle formula in Theorem \ref{general}, is enough to show a relationship between the total space and base space of an elliptic fibration through a birational Fujita-Zariski decomposition.


\subsection{Generalized Zariski Decompositions for Elliptic Fibrations}

\smallskip
{In the following, we establish the compatibility of the Fujita-Zariski decomposition with elliptic fiber spaces which extends the arguments of \cite{Fujita86} for elliptic threefold to higher dimensions. Furthermore, we will show the explicit decomposition in the case where we have existence of log minimal models for the base of the fiber space.}

\begin{theorem}
\label{EllFZD}
Given an elliptic fibration $X_0 \rightarrow B_0$, there exist a birationally equivalent fibration $X \rightarrow B$ and a $\mathbb{Q}$-divisor $\Lambda$ on $B$ such that $K_X$ birationally admits a Fujita-Zariski decomposition if and only if $K_B + \Lambda$ birationally admits a Fujita-Zariski decomposition where $X$ and $(B, \Lambda)$ are as in Lemma \ref{HironakaFlat}.
\end{theorem}

\begin{proof}
Without loss of generality we can assume that $X_0$ and $B_0$ are smooth, with ramification divisor $\Lambda_0$ having simple normal crossing. As in Lemma \ref{HironakaFlat} we have:
\[
\begin{tikzcd}
\displaystyle
X_0 \arrow[d, "\pi_0"']& X_1 \arrow[l]\arrow[d]& \arrow[bend right=30,swap]{ll}{\nu}  X \arrow[l]\arrow[d, "\pi"]\\
B_0 & B_1\arrow[l]& B \arrow[l]\\
\end{tikzcd}
\]
where all the horizontal maps are birational morphisms, $X_1$ is the resolution of the flattening of $\pi_0$ and 
 and $\pi:X \rightarrow B$ is as in Theorem \ref{general}. We have
$
K_X = \pi^*(K_B + \Lambda) + E - G
$
where $(B,\Lambda)$ is a klt pair of dimension $n - 1$.

Assume that $K_X$ birationally admits a Fujita-Zariski decomposition. Without loss of generality, we assume that $K_X$ admits a Fujita-Zariski decomposition, in the sense of FZ-A, equivalently, FZ-B, as in Definition \ref{ZD} and Remark \ref{FZAB}.
Then we have 

\[
P + N = K_X= \pi^*(K_B + \Lambda) + E - G
\]
with $P, \ N$ as in Definition \ref{ZD}.
We will 
show that $K_B + \Lambda$ birationally admits a Fujita-Zariski decomposition. 

  We have $G$ is a $\nu$-exceptional effective divisor, since $X_1 \rightarrow B_1$ is equidimensional, as it is a flat morphism over a smooth base, and  $\codim (\pi(G)) \geq 2$.
   Furthermore $K_{X} = \nu^*(K_{X_0}) + F$, with $F$ an effective $\nu$-exceptional divisor, since $X$ and $X_0$ are smooth. Then  $F$ is numerically fixed by $K_X$ and $F + G$ is numerically fixed by $\nu^*(K_{X_0}) + F + G = K_X + G$, \cite[Prop. 1.10]{Fujita86}. Since $F$ is numerically fixed by $K_X$ and $K_X = P + N$ is a Fujita-Zariski decomposition then $K_X - F = \nu^*(K_{X_0})$ has a Fujita-Zariski decomposition by Lemma \ref{NumFix}.  Similarly, since $F + G$ is numerically fixed by $\nu^*(K_{X_0}) + F + G = K_X + G$,
   thus $K_X + G$ admits a Fujita-Zariski decomposition. In both cases $P$ is the nef part of the decomposition. It follows
   that
\[
\pi^*(K_B + \Lambda) + E = P + N + G
\]
is a Fujita-Zariski decomposition. 
Since $E$ is also numerically fixed by $\pi^*(K_B + \Lambda) + E$ (Theorem \ref{general} and \cite[Prop 1.10]{Fujita86}); then $\pi^*(K_B + \Lambda)$ also admits a Fujita-Zariski decomposition (Lemma \ref{NumFix}). Then $K_B + \Lambda$ birationally also admits a Fujita-Zariski decomposition by  Proposition \ref{ExToZariski}.
\smallskip

Assume now that $K_B + \Lambda$ birationally admits a Fujita-Zariski decomposition. Without loss of generality we assume that $K_B+ \Lambda = P_\Lambda + N_\Lambda$ is a Fujita-Zariski decomposition. 
We have $\pi^*(K_B + \Lambda) = \pi^*(P_\Lambda) + h^*(N_\Lambda)$ is then a  Fujita-Zariski decomposition (Proposition \ref{ExToZariski}),  with $\pi^*(P_\Lambda)$ the nef portion of the decomposition. The canonical bundle formula 
$K_X = \pi^*(K_B + \Lambda) + E - G$
and \cite[Prop 1.10]{Fujita86} imply that $E$ is numerically fixed by $\pi^*(K_B + \Lambda) + E$. We have then a Fujita-Zariski decomposition for $K_X + G =\pi^*(K_B + \Lambda) + E = \pi^*(P_\Lambda) + \pi^*(N_\Lambda) + E$ with nef part $\pi^*(P_\Lambda)$. Similarly, with Lemma \ref{NumFix} applied to $G$,  we deduce that $K_X$ admits a Fujita-Zariski decomposition of the form 
\[
K_X=\pi^*(P_\Lambda) + \pi^*(N_\Lambda) + E - G. 
\]
Here $\pi^*(N) + E-G$ is effective and $\pi^*(P_\Lambda)$ is nef.

\end{proof}

\begin{theorem}\label{NewFZ}
Let   $\pi_0:X_0 \rightarrow B_0$ be an  elliptic fibration, $\dim X_0=n$ and $\kappa(X_0) \geq 0$. 
Assume the existence of minimal models for $klt$ pairs {of non negative Kodaira dimension} in dimension $n-1$. There exist  birationally equivalent fibrations and birational morphisms 	$\phi_{\tilde B}$ and $\phi_B$
{\small\[
\begin{tikzcd}
X_0 \arrow[dd, "\pi_0"']&  & X \arrow[d,"\tilde{\pi}"] \arrow[ll] \arrow[ldd, "\pi"'] \arrow[rdd, "\epsilon"]&\\
&&\tilde{B} \arrow[ld, "\phi_B"'']\arrow[rd,"{\phi_{\bar B}}"']&\\
B_0 &  (B, \Lambda) \arrow[l]\arrow[rr,dashed] && (\bar{B},\bar{\Lambda})\\
\end{tikzcd}
\]
}
\noindent such that $
K_{X} = \epsilon^*(K_{\bar{B}} + \bar{\Lambda}) + \tilde{\pi}^*\Gamma + E - G $ is a Fujita-Zariski decomposition of  
$K_X= \pi^*(K_B+\Lambda) + E-G$

where 
\begin{itemize}
	\item    $(\bar{B},\bar{\Lambda})$ is a log minimal model of the klt pair $(B,\Lambda)$ 
	\item $\Gamma$ is an $\phi_{\tilde B}$-exceptional effective $\mathbb{Q}$-divisor.
	\item $P = \epsilon^*(K_{\bar{B}} + \bar{\Lambda})$  is the nef part and   $N = \tilde{\pi}^*\Gamma + E - G$ the effective part of the Fujta-Zariski decomposition.
	
\end{itemize}
\end{theorem}

\begin{proof}

As in  Theorem \ref{EllFZD} we have the  birationally equivalent fibrations:
\[
\begin{tikzcd}
X_0 \arrow[d, "\pi_0"']& X_1 \arrow[l]\arrow[d]& X \arrow[l]\arrow[d, "\pi"]\\
B_0 & B_1\arrow[l]& B \arrow[l]\\
\end{tikzcd}
\]
 and  $K_X = \pi^*(K_B + \Lambda) + E - G$,
where $(B,\Lambda)$ is a klt pair of dimension $n - 1$.
By the hypotheses $ 0 \leq \kappa(X) = \kappa(B, K_B + \Lambda) $  (Proposition \ref{kodEff}) and existence of minimal models for klt pairs of dimension $n-1$, $(B, \Lambda)$ has log minimal model  $(\bar{B}, \bar{\Lambda})$.
 Let $\tilde{B}$ be a common log resolution of $(B,\Lambda)$ and $(\bar{B}, \bar{\Lambda})$ and $\tilde{X}$ be a resolution of $X \times_B \tilde{B}$. 
 As in Theorem \ref{EllFZD}
  we can assume  without loss of generalities $\tilde{X}=X$. We have the following commutative diagram:
{\small {\[
\begin{tikzcd}
X_0 \arrow[dd, "\pi_0"']&  & X \arrow[d,"\tilde{\pi}"] \arrow[ll] \arrow[ldd, "\pi"'] \arrow[rdd, "\epsilon"]&\\
&&\tilde{B} \arrow[ld, "\phi_B"'']\arrow[rd,"{\phi_{\bar B}}"']&\\
B_0 &  (B, \Lambda) \arrow[l]\arrow[rr,dashed] && (\bar{B},\bar{\Lambda})\\
\end{tikzcd}
\]
}}
 By the Negativity Lemma, \cite[Lemma 3.39]{KollarMori},  we have $\phi_B^*(K_B + \Lambda) = \phi_{\bar B}^*(K_{\bar{B}} + \bar{\Lambda}) + \Gamma$ with $\Gamma$ effective and $\phi_{\bar B}$-exceptional. From the arguments of \cite[Thm. 1.5]{Birkar1}, $\phi_{ B}^*(K_B + \Lambda) = \phi_{\bar B}^*(K_{\bar{B}} + \bar{\Lambda}) + \Gamma$ is a Fujita-Zariski decomposition of $\phi_{ B}^*(K_B + \Lambda)$ with $\phi_{\bar B}^*(K_{\bar{B}} + \bar{\Lambda})= P_\Lambda$  the nef part and $\Gamma=N_\Lambda$.  Then $K_B + \Lambda$ birationally admits a Fujita-Zariski decomposition and so by the arguments of Theorem \ref{EllFZD}  
we have that:

\begin{align*}
K_X &= \tilde{\pi}^*h^*(K_{\bar{B}} + \bar{\Lambda}) + \tilde{\pi}^*(\Gamma) + E - G\\
&=\epsilon^*(K_{\bar{B}} + \bar{\Lambda}) + \tilde{\pi}^*(\Gamma) + E - G.
\end{align*}

\end{proof}

\begin{corollary}\label{CKM} Under the assumption of the hypothesis and notation of Theorem \ref{NewFZ}, the canonical model of $X$ is isomorphic to the log canonical model of $(B,\Lambda)$.
\end{corollary}

\begin{proof}
A Fujita-Zariski decomposition is a CKM-Zariski decomposition (\ref{ZariskiRelation}).
In Theorems \ref{EllFZD} and \ref{NewFZ}, we showed that $P= \pi^*(P_\Lambda)$.  Then, up to a change in grading, the canonical rings of $X$ and $(B,\Lambda)$ are isomorphic and the canonical models are isomorphic.
\end{proof}

\medskip

\subsection{The Zariski Decompositions and  Minimal Models for Elliptic Fibrations}~
\smallskip

We now use our results of Zariski decomposition and elliptic fibrations (Theorems \ref{NewFZ}  and \ref{RelMM}) to give a different proof of part (2) in Theorem \ref{MM}. Note that the statement is stronger. In particular, $\bar{B}$ is $\mathbb{Q}$-factorial.


\begin{theorem}
\label{NewMain}
 Let  $\pi_0:X_0 \rightarrow B_0$ be elliptic fibration, with  $\dim (X)= n$ and $\kappa(X) \geq 0$. 

Assume one of the following:
\begin{enumerate}
\item Log minimal models for $klt$ pairs {of non negative Kodaira dimension} in dimension $n-1$ exist.
\item Any sequence of flips for generalized klt pairs of dimension at most $n-2$ terminates and 
$K_B+ \Lambda$ admits a weak Zariski decomposition.
\end{enumerate}

There exists a birationally equivalent fibration $\bar{\pi}: \bar{X} \rightarrow \bar{B}$ such that 
\begin{itemize}
	\item $\bar{B}$ is normal and $\mathbb{Q}$-factorial.
	\item There exists a effective divisor $\bar{\Lambda}$ on $\bar{B}$ such that $(\bar{B}, \bar{\Lambda})$ is a klt pair.
	\item $\bar{X}$ has at worst terminal singularities.
	\item $K_{\bar{X}} \equiv_{\mathbb{Q}} \bar{\pi}^*(K_{\bar{B}} + \bar{\Lambda})$
	\item $K_{\bar{X}}$ is nef
\end{itemize}
\end{theorem}

\begin{proof}
Assumption 
(2) ensures the existence of a minimal model for $(B,\Lambda)$
 \cite[Thm. 1]{HaconMoraga}. 
Let $(\bar{B},\bar{\Lambda})$ be a minimal model as in  Theorem \ref{NewFZ}. We have the following diagram:
{\small\[
\begin{tikzcd}
X_0 \arrow[dd, "\pi_0"']&  & X \arrow[d,"\tilde{\pi}"] \arrow[ll] \arrow[ldd, "\pi"'] \arrow[rdd, "\epsilon"]&\\
&&\tilde{B} \arrow[ld, "\phi_B"'']\arrow[rd,"{\phi_{\bar B}}"']&\\
B_0 &  (B, \Lambda) \arrow[l]\arrow[rr,dashed] && (\bar{B},\bar{\Lambda})\\
\end{tikzcd}
\]
}
and the following Fujita-Zariski decomposition of the canonical divisor of $K_X$:
\[
K_{X} = \epsilon^*(K_{\bar{B}} + \bar{\Lambda}) + \tilde{\pi}^*\Gamma + E - G, \\ \epsilon^*(K_{\bar{B}} + \bar{\Lambda}) \text{ nef and  }  \tilde{\pi}^*\Gamma + E - G \ \text{effective.}
\]

We apply Theorem \ref{RelMM} to the relative MMP with respect to $\epsilon: X \rightarrow \bar{B}$:

\[
\begin{tikzcd}
& X \arrow[d,"\tilde{\pi}"] \arrow[ldd, "\pi"'] \arrow[rdd, "\epsilon"] \arrow[r, dashed]& \bar{X} \arrow[dd, "\bar{\pi}"]\\
&\tilde{B} \arrow[ld, "\phi_B"'']\arrow[rd,"\phi_{\bar B}"']&\\
(B, \Lambda) \arrow[rr,dashed] && (\bar{B},\bar{\Lambda})\\
\end{tikzcd}
\]

To apply Theorem \ref{RelMM}
we want to show that no component of the effective divisor $\tilde{\pi}^*\Gamma + E - G$ is a pullback of some $\mathbb{Q}$-divisor on $\bar{B}$. 

It is sufficient to show that no component of $\tilde{\pi}^*\Gamma$ and $E$ contains the pullback of a divisor on $\bar{B}$, since they contain all the components of $\tilde{\pi}^*\Gamma + E - G$. We have that $\Gamma$ is contracted by $\phi_B$ 
thus   $\tilde{\pi}^*\Gamma$ cannot contain the pullback of a divisor on $\bar{B}$. The components of $E$ can map down to a space of codimension $1$ or to a space of codimension $\geq 2$ on $\bar{B}$.

 We then need to show that when $\epsilon_*(D)$ has  codimension one  in $\bar B$, then $D$ does not contain the fiber over the points in its image on $\bar{B}$.

Assuming that to be the case, $\tilde \pi_* ({D})$ will be an effective divisor. 
 Furthermore, $\tilde \pi_* ({D})$ cannot be contracted by $\phi_B:\tilde{B} \rightarrow B$, because then it would mean
$D$ would map  to a space of codimension $\geq 2$ on $B$ and since $(\bar{B},\bar{\Lambda})$ is a log minimal model of $(B,\Lambda)$,  $\tilde \pi_*({D})$ would also be contracted by $\phi_{\tilde B}$.  Since $\tilde \pi_*({D})$ is not contracted by $\phi$, then 
$D$ is exceptional, in the sense of Theorem \ref{general}; in particular
 $D$ does not contain preimage of general points on its image in $\tilde B$ and  $D$ is not a pullback of a divisor on $\tilde{B}$ and a fortiori of $\bar{B}$ also. 

By Theorem \ref{RelMM}, we will have $K_{\bar{X}} \equiv_\mathbb{Q} \bar{\pi}^*(K_{\bar{B}} + \bar{\Lambda})$ and $K_{\bar{X}}$ is nef since it is numerically the pullback of a log canonical divisor of a log minimal model. $\bar{X}$ has at worst terminal singularities since it is obtained from running a relative MMP on a smooth variety. 
\end{proof}

\section{Applications}
\subsection{Existence of Zariski decompositions and minimal models}\label{Existence}
\begin{corollary}\label{FZD}
Assume the existence of minimal models for klt pairs in dimension $n-1$ with non-negative Kodaria dimension. Given an elliptic $n$-fold, $\pi: X \rightarrow B$, then we have that $K_X$ birationally admits a Fujita-Zariski decomposition.
\end{corollary}

\begin{corollary}
 Let  $\pi: Y \rightarrow T$ be  and elliptic fibration with  
 $\dim (Y) = n$ and $\kappa(Y) \geq 0$ .  If generalized $klt$ flips terminate in dimension up to $n-1$, then any minimal model program for $Y$ terminates.
\end{corollary}
\begin{proof}
Theorem \ref{NewFZ} establishes a weak Zariski decomposition for $K_X$ and the results follow from \cite[Thm. 1]{HaconMoraga}.
\end{proof}

Since minimal model exist for $klt$ pairs of non-negative Kodaira dimension of dimension up to $4$ we have the following:
\begin{corollary}
\label{EllFib5}
An elliptically fibered variety of dimension $n \leq 5$  with non-negative Kodaira dimension has  a birationally equivalent fibration $\bar{\pi}:\bar{X} \rightarrow \bar{B}$ where $\bar{X}$ is a minimal model and $K_{\bar{X}} \equiv_{\mathbb{Q}} \bar{\pi}^*(K_{\bar{B}} + \bar{\Lambda})$.
\end{corollary}

\begin{theorem}
Assume termination of flips for dlt pairs in dimension $n - 2$. Let $X \rightarrow B$ and $(B, \Lambda)$ as in Lemma \ref{HironakaFlat}. $X$ has a minimal model if and only if $(B,\Lambda)$ has a log minimal model.
\end{theorem}
\begin{proof}
 $K_X$ birationally admits a Fujita-Zariski decomposition if and only if $K_B + \Lambda$ birationally admits a Fujita-Zariski decomposition (Theorem \ref{EllFZD}).
 If $(B,\Lambda)$ has a log minimal model then following the argument in the proof of   Theorem \ref{NewMain} we can construct a minimal model of $X$.  
 
 If $X$ has a minimal model, the arguments of \cite[Thm. 1.5]{Birkar1} show that $K_X$ birationally admits a Fujita-Zariski decomposition. Then $K_B + \Lambda$ birationally admits a Fujita-Zariski decomposition (Theorem \ref{EllFZD}).
  Now since  $\dim(B) = n - 1$, $(B,\Lambda)$ has a log minimal model \cite[Thm. 1.5]{Birkar1}.
\end{proof}

\medspace
\subsection{Abundance and Elliptic Fibrations}\label{Abundance}~

\smallskip

{In the previous section we proved  the compatibility of the Fujita-Zariski decomposition with elliptic fibrations and minimal models. Now we turn our attention to good minimal models. Associated to a good minimal models we have the Nakayama-Zariski decomposition (Definition \ref{ZD}). 
We prove that the Fujita-Zariski decompositon when there is  a good minimal model is also a Nakayama-Zariski decomposition.}

\begin{corollary}
\label{NZ} 
Let   $\pi_0:X_0 \rightarrow B_0$ be an  elliptic fibration, $\dim X_0=n$ and $\kappa(X_0) \geq 0$. 
Assume the existence of good minimal models for $klt$ pairs {of non negative Kodaira dimension} in dimension $n-1$.  Then the Fujita-Zariski decomposition in Theorem \ref{NewFZ}  is also a Nakayama-Zariski decomposition.
\end{corollary}
\begin{proof}
Using the notation and set up as in Theorem \ref{NewFZ}, the Fujita-Zariski decomposition of $K_X$ is given by:
\[
K_X =\epsilon^*(K_{\bar{B}} + \bar{\Lambda}) + \tilde{\pi}^*(\Gamma) + E - G
\]
By assumption $K_X$ is pseudoeffective and so it also has a Nakayama-Zariski decomposition:
\[
K_X = P_\sigma(K_X) + N_\sigma(K_X);
\]
 we will show that $\epsilon^*(K_{\bar{B}} + \bar{\Lambda}) = P_\sigma(K_X)$ and $\tilde{\pi}^*(\Gamma) + E - G = N_\sigma(K_X)$.

As $(\bar{B},\bar{\Lambda})$ is a good minimal model, $K_{\bar{B}} + \bar{\Lambda}$  is semiample. From the arguments in Theorem \ref{NewFZ}, we have that $\tilde{\pi}^*(\Gamma) + E - G$ is $\epsilon$-degenerate thus by \cite[Lemma 2.16]{GongyoLehmann} we have:
\begin{align*}
\tilde{\pi}^*(\Gamma) + E - G &\leq N_\sigma(K_X)\\
P_\sigma(\epsilon^*(K_{\bar{B}} + \bar{\Lambda})) &= P_\sigma(K_X) 
\end{align*}
From \cite[Lemma 2.9]{GongyoLehmann}, we have that for any pseudoeffective divisor $D$, we have $N_\sigma(D)$ is contained in $B_ -(D)$ where:
\[
B_-(D) = \bigcup_{\epsilon > 0}Bs(D + \epsilon A)
\]
and $Bs(F)$ denotes the base locus of $F$ and $A$ is any ample divisor. The definition is independent of the choice of $A$. Now as we have that $\epsilon^*(K_{\bar{B}} + \bar{\Lambda})$ is semiample, we must have that $B_-(\epsilon^*(K_{\bar{B}} + \bar{\Lambda})) = \emptyset$, so that $N_\sigma(\epsilon^*(K_{\bar{B}} + \bar{\Lambda})) = 0$. This implies that:
\[
P_\sigma(K_X) = \epsilon^*(K_{\bar{B}} + \bar{\Lambda})
\]
and 
\[
N_\sigma(K_X) = \tilde{\pi}^*(\Gamma) + E - G.
\]
\end{proof}

\begin{corollary}\label{ifAbundance}
[Proposition \ref{kodEff}, Theorem \ref{MM}]
Let $\pi: Y \rightarrow T $ be an elliptic fibration
\begin{enumerate}
\item If  $\dim (Y) = 4$ there exists a birational equivalent fibration $\bar{X} \to \bar B$, $\bar{X}$ with $\Q$-factorial  terminal singularities, $(\bar B, \bar \Lambda)$ with  klt singularities such that $K_{\bar X}  \equiv \bar{\pi}^*(K_B + \bar \Lambda)$ and either 
$Y$ is birationally a Mori fiber space or  $\bar{X}$ is a good minimal model.
\item 
 If $\kappa(Y) = n - 1$, there exists  a birationally equivalent fibration $\bar{\pi}:\bar{X} \rightarrow \bar{B}$ such that  $\bar{X}$ is a good minimal model, $K_{\bar{X}} \equiv_{\mathbb{Q}} \bar{\pi}^*(K_{\bar{B}} + \bar{\Lambda})$ and $(\bar{B},\bar{\Lambda})$ has klt singularities.
\end{enumerate}
\end{corollary}

\begin{proof} It follows from Proposition \ref{kodEff}, Theorem \ref{kltmin} and Theorem \ref{MM}. See also  {\cite[Thm. 4.4]{Lai2011}}, {\cite[Cor. 4.5]{GongyoLehmann}}.
\end{proof}

\section{The dimension of the fibers and equidimensionality, up to birational equivalence}\label{equid}
   While it is easy to fabricate  examples of minimal elliptic threefolds which are not equidimensional starting from ones which are,  many  smooth Calabi-Yau threefolds have  a natural  elliptic fibration which is not equidimensional. These  examples  were  mostly found during searches to provide evidence that very large classes of Calabi-Yau threefolds are birationally  elliptically fibered  \cite{BraunV2013, HuangTaylorFibrations2019}. 
 
 If $\dim (X) =3$, there  always exists a birational equivalent elliptic fibration, minimal or  a Mori fiber space, which is equidimensional \cite[Cor. 8.2]{Grassi95}.  
By contrast  there is an example of a non-equidimensional elliptic fourfold  for which it is not known if an  equidimensional model exists \cite{codimthree}.
Examples of local Calabi-Yau fourfolds in  generalized Weierstrass form  with  possibly non-equidimensional elliptic fibrations are also described in \cite{LawrieSNSteroids}.
 In the  example  in \cite{codimthree} a particular fiber contains a smooth surface and the fibration is otherwise  equidimensional.  
 Corollary \ref{equidim3}  proves that this is what it can  be generally expected.



Theorem \ref{KawNakayama} and Corollary \ref{equidim3} are stronger than what one could obtain from a log-minimal model run, in at least two aspects. First,  the singularities  in our models  are terminal, while a minimal model run gives log terminal singularities, as  in Example \ref{Ex1.1}.   See  \cite{GrassiWeigand} for an analysis of  terminal versus log-terminal  singularity in this context. In addition we prove that not only there are no exceptional divisors in the fibers outside a codimension $3$ set, but the fibration is  equidimensional there.

\begin{theorem}
\label{KawNakayama}
Let $f:X \rightarrow B$ be an elliptic fibration such that $X$ has at worse  $\mathbb{Q}$-factorial  terminal singularities and $K_X = f^*(L)$ where $L$ is a $\mathbb{Q}$-Cartier divisor on $B$. Then there exists a birationally equivalent elliptic fibration $h: Y \rightarrow T$ and a $\Q$-divisor $\Lambda_T$ such that:
\begin{enumerate}
	\item $Y$ has at worse  $\mathbb{Q}$-factorial  terminal singularities.
	\item $K_Y = h^*(K_T+ \Lambda_T)$ where $(T, \Lambda_T )$ is klt. 
	\item There is no  effective divisor $E $ in $Y$ such that $\codim h(E) \geq 2$.
\end{enumerate}
\end{theorem}
\begin{proof}
Let $\rho(X/B)$ be the rank of the relative Neron-Severi group of $f:X \rightarrow B$; it  has finite rank and we proceed by induction on this invariant. If $E$ is an effective divisor on $X$ such that $f(E)$ has codimension $\geq 2$, then we can take an effective Cartier divisor without fixed component, $C$, on $B$ that contains $f(E)$ and we have that:
\[
f^*(C) = D + F
\]
where $F$ is the maximal component of $f^*(C)$ such $\codim (\pi(F)) \geq 2$ and $D = f^*(C) - F$. Then $Supp(E) \leq Supp(F)$ and $\codim h(D)=1.$
 $K_X + D$ is $f$-nef  if and only if  $D$ is $f$-nef.    
 
If $D$ is not $f$-nef,  the $f$-minimal program on the log pair $(X, \epsilon D)$ for $0 < \epsilon \ll 1$ 
produces  a relatively minimal pair $(Y, \epsilon D')$ over $B$  \cite[Thm. 2.12]{HaconXu}. 
Furthermore $D$ is $f$-movable \cite[Def. 1.1]{Kawamata97} so that running the relative log minimal model program on $(X, \epsilon D)$ over $B$ results in a sequence of $D$-flips. As $K_X$ is numerically trivial over $B$, we have that this sequence of $D$-flips is a sequence of flops. We have the diagram:
\[
\begin{tikzcd}
(X, \epsilon D) \arrow[d, "f"'] \arrow[r,dashed, "\phi"]& (Y, \epsilon D') \arrow[dl, "g"']\\
B &\\
\end{tikzcd}
\]
where $K_Y = g^*(L)$, $D' = \phi_*D$ and $K_Y + \epsilon D'$ is $g$-nef. So if $D$ was not $f$-nef we can obtain a birational model $Y$ byt a sequence of flops and we can reduce to the case where we have $g:Y \rightarrow B$ with a $g$-nef dvisor $D'$.

If $D'$ is $g$-nef, then $D'$ is $g$-semiample  \cite[Thm. A.4]{Nakayama2002}, and there exists a morphism $h:Y \rightarrow T$ that factors as follows:
\[
\begin{tikzcd}
X \arrow[d, "f"'] \arrow[r,dashed, "\phi"]& Y \arrow[dl, "g"'] \arrow[d, "h"]\\
B & T \arrow[l, "\psi"]\\
\end{tikzcd}
\]
Since  $K_Y = g^*(L) = h^*(\psi^*(L))$,  and by letting $L_Y = \psi^*(L)$  we have that $K_Y=h^*(L_Y)$,  $L_Y$ $\mathbb{Q}$-divisor on $T$. Furthermore, $\psi$ is not an isomorphism since $D'$ is numerically trivial over $T$ but not over $B$. $\dim(T)=\dim (B)$
since $D' \leq g^*(C)$ and so is not $f$-ample. This implies that $\rho(Y/T) < \rho(X/B)$ and that  $h:Y \rightarrow T$ also satisfies the hypothesis of the theorem. As the rank of the Neron-Severi group is finite, 
 this process must eventually terminate. 
\end{proof}

\begin{corollary}\label{equidim3}
\label{Equidimension}
Let $\pi_0: X_0 \rightarrow B_0$ be an elliptic fibrations as in Theorem \ref{MM} or Theorem \ref{NewMain}, then there exists a birationally equivalent elliptic fibration $h: Y \rightarrow T$, where $Y$ is a relatively minimal model over $T$ and $h$ is equidimensional over an open set $U \subset T$ whose complement has codimension $\geq 3$. If furthermore we have that $\kappa(K_{X_0}) \geq 0$, we can take $Y$ to be a minimal model.
\end{corollary}

\begin{proof}
From Theorem \ref{MM} or \ref{NewMain}, we obtain a birationally equivalent elliptic fibration $\bar{\pi}: \bar{X} \rightarrow \bar{B}$ that satisfies the hypothesis of Theorem \ref{KawNakayama},  namely   a birationally equivalent fibration $h: Y \rightarrow T$ such that there is no effective divisor $E $ in $Y$ such that $\codim h(E) \geq 2$.
If furthermore  $\kappa(K_{X_0}) \geq 0$, we can take $\bar{X}$ to be a minimal model and, since $Y$ is obtained via a sequence of flops from $\bar{X}$, we have that $Y$ is also a minimal model.

 We will prove that the  general fibers of $h$ over subvarieties of codimension $\leq 2$ are $1$-dimensional. 
Let $S \subset T$ be an irreducible closed subvariety of $T$. If $S$ has codimension $1$ then we have that $h^{-1}(S)$ has codimension $1$,  and a general fiber over $S$ is $1$ dimensional. Let  now $S$ be of codimension $2$. Then $h^{-1}(S)$ has codimension $\leq 2$. Since no divisors of $Y$ maps down to space of codimension $2$, we must have that $h^{-1}(S)$ has codimension $2$. By counting the dimensions, we have that the general fibers over $S$ is
$1$ dimensional. Thus general fibers of $h$ over subvarieties of codimension $\leq 2$ are $1$-dimensional. Thus, we have that $h$ is equidimensional over some open set $U \subset T$ whose complement has codimension $\geq 3$.
\end{proof}

\bibliographystyle{plain}
\bibliography{refs}{}

\end{document}